\documentclass[a4paper,12pt]{article}

\usepackage{amsfonts}
\usepackage{amsmath}
\usepackage{amssymb}
\usepackage{booktabs}
\usepackage{graphicx}
\usepackage{color}
\usepackage{enumerate}
\usepackage{float}

\newenvironment{proof}{\noindent {\bf Proof:}}{\hfill $\Box$}

\newtheorem{theorem}{Theorem}
\newtheorem{lemma}{Lemma}

\newtheorem{corollary}{Corollary}

\newtheorem{problem}{Problem}

\def\vol{\mathrm{vol}}

\def\one{\mathbb{I}}

\usepackage{amsmath}
\usepackage{amssymb}
\usepackage{amsfonts}
\usepackage{graphicx}
\usepackage{eufrak}

\newcommand{\K}{\mathcal{K}}
\newcommand{\B}{\mathcal{B}}
\newcommand{\U}{\mathcal{U}}

\newcommand{\Pd}{{\mathcal P}_d}
\newcommand{\unif}[1]{ {\mathbb{U}}_{#1} }

\title{\bf Uniform sample generation\\ in semialgebraic sets}

\begin{document}

\author{Fabrizio Dabbene$^1$, Didier Henrion$^{2,3,4}$, Constantino Lagoa$^{5}$}

\footnotetext[1]{CNR-IEIIT; c/o Politecnico di Torino; C.so Duca degli Abruzzi 24, Torino; Italy. {\tt fabrizio.dabbene@polito.it}}
\footnotetext[2]{CNRS; LAAS; 7 avenue du colonel Roche, F-31400 Toulouse; France. {\tt henrion@laas.fr}}
\footnotetext[3]{Universit\'e de Toulouse; LAAS; F-31400 Toulouse; France.}
\footnotetext[4]{Faculty of Electrical Engineering, Czech Technical University in Prague,
Technick\'a 2, CZ-16626 Prague, Czech Republic.}
\footnotetext[5]{Electrical Engineering Department, The Pennsylvania State University, University Park, PA 16802, USA. {\tt lagoa@engr.psu.edu}}

\date{Draft of \today}

\maketitle

\begin{abstract}
We propose efficient techniques for generating independent identically distributed uniform random samples inside
semialgebraic sets. 
The proposed algorithm leverages recent results on the approximation of indicator functions by polynomials 
to develop acceptance/rejection based sample generation algorithms with guaranteed performance in terms of rejection rate (the number of samples that 
should be generated in order to obtain an accepted sample).
Moreover, the  {acceptance} rate is shown to be is asymptotically optimal, in the sense that it tends to one (all samples accepted) 
as the degree of the polynomial approximation increases.
The performance of the proposed method is illustrated by a numerical example.
\end{abstract}

\section{Introduction}

Generating independent uniformly distributed samples over ``simple,'' sets such as $\ell_p$ norms, has been thoroughly  studied and algorithms are available which efficiently solve this problem; e.g., see~\cite{TeCaDa:13}. However, efficient generation of  independent  {uniformly distributed} samples over general sets remains a difficult problem.  { This especially true in the case where the set is of low volume, hard to localize and/or  nonconvex or even disconnected.}

Although an interesting problem on its own, efficient sample generation can be used for solving complex analysis and design problems. Examples of these are chance-constrained optimization,  robust optimization  and multi-objective controller design,
see e.g. \cite{TeCaDa:13} and the many references therein.

In this paper, we present some preliminary results that aim at solving the sample generation problem over semialgebraic sets with non-empty interior. More precisely, we leverage recent results on polynomial  approximation of indicator functions of sets defined by polynomial inequalities~\cite{DabHen:13} to develop algorithms that
\begin{enumerate}[i)]

\item Generate independent uniformly distributed samples over the given semialgebraic set;

\item Have  provable bounds on sample rejection rate;

\item Have a rejection rate tends to zero as the degree of the polynomial approximation of the indicator function increases.

\end{enumerate}

The problem of random sample generation has been the subject of active research. Techniques for univariate generation techniques are discussed e.g.\ in \cite{Devroye:86}. These methods, however, are not readily extendable to the sets usually encountered in robust control. Hence, specific techniques  for generating uniform (or radially symmetric) samples in the
$\ell_{p}$ vector norm ball are discussed in \cite{TeCaDa:05}.
In the papers \cite{CalDab:02} and  \cite{CaDaTe:00}
methods for random sample generation in real and complex spectral norm balls are developed; see also \cite{ZhoFen:06}
for generation of matrices having a Toeplitz structure. 
The generation of causal stable dynamic operator has been the subject of different studies.
In \cite{ShcDab:11} various techniques for generating random samples inside the Schur stability region are discussed, while  \cite{SzLaMa:05} presents an algorithm for approximately generating uniform transfer functions in the $\mathcal{RH}_{\infty}$ ball.

We remark that the methods discussed in this paper are non-asymptotic,
contrary to the Markov chain Monte Carlo techniques discussed for instance 
in \cite{Bollobas:97,Lovasz:96} and references therein. In these papers,  the desired distribution
is obtained by simulating a random walk on a graph. and the independence and uniformity properties are only
guaranteed asymptotically. 
Therefore, the methods discussed in this paper can be implemented on
parallel and distributed architectures, see e.g.\ \cite{Fujimoto:00}.

The paper outline is as follows.
Section~\ref{sec:statement} provides a precise definition of the problem addressed in this paper and describes some auxiliary results needed for the development of the proposed approach. The algorithm for sample generation is described in Section~\ref{sec:unifgen}.  {Since this algorithm requires generation of samples from a distribution with polynomial density, details on how one can do this efficiently are given in Section~\ref{sec:pol_sample}.} In Section~\ref{sec:examples}, we illustrate the performance of the proposed approach with a few academic examples. Finally, in Section~\ref{sec:conclusion}, concluding remarks are presented and further research directions are delineated.

\section{Problem statement}
\label{sec:statement}

Consider a compact basic semialgebraic set described by polynomial inequalities
\begin{equation}
\label{Kset}
\K:=\{x \in {\mathbb R}^n : g_i(x) \geq 0, \:i=1,2,\ldots,m\}
\end{equation}
where $g_i(x)$, $i=1,\ldots,m$ are given real multivariate polynomials.
The problem we consider is the following:\\

\begin{problem} {\it Given the semialgebraic  set $\K$ defined in (\ref{Kset}),
generate $N$ independent identically distributed (i.i.d.) random samples $x^{(1)},\ldots,x^{(N)}$ uniformly distributed in $\K$.}
\end{problem}

Note that the considered setup is very general, and encapsulates many problems of interest to the control community. In particular, the algorithm presented in this paper can be used to generate  uniform samples
in the solution set of linear matrix inequalities (LMIs). Indeed, it is a well-known fact that LMI sets are (convex)
basic semialgebraic sets.
To see this, consider the LMI set
\[
\{x \in {\mathbb R}^n : F(x)=F_{0}+F_{1}x_{1}+\cdots +F_{n}x_{n}\succeq 0\}
\]
where $\succeq 0$ stands for positive semidefinite and the matrix $F(x)$
has size $m\times m$.
A vector $x$ belongs to the LMI set if and only if all the principal minors of $F(x)$ are nonnegative. This immediately leads to a set 
of $m$ polynomial inequalities in $x$.

\subsection{Preliminaries}
We define by $\Pd$ the vector space of multivariate real polynomials in $n$ variables of degree less than or equal to $d$.
The uniform density over a set $\K$ of nonzero volume is defined as
\begin{equation}
\unif{\K}(x) := \frac{\one_\K(x)}{\vol(\K)}, \label{eq:uniform}
\end{equation}
where $\one_\K(x)$ denotes the indicator function of the set $\K$
\[
\one_\K(x)=
\begin{cases}
1 & \text{if } x\in \K\\
0 & \text{otherwise,} 
\end{cases}
\]
and $\vol(\K)$ is the Lebesgue measure (volume) of $\K$; e.g., see  \cite{Halmos:50} for details on Lebesgue measures and integration. 
The idea at the basis of the method we propose is to find a suitable approximation of the set $\K$, using the framework introduced in
\cite{DabHen:13}. To this end, let us consider a polynomial of degree $d$
\[
p(x)=\sum_{\alpha \in {\mathbb N}^n, |\alpha|\leq d} p_{\alpha} \pi^{\alpha}(x)
\]
where the sum ranges over all integer vectors of size $n$ summing up to $d$ or less.
Let us now introduce the  polynomial super-level set
\[
\U(p):=\{x \in {\mathbb R}^n : p(x) \geq 1\}.
\]
The use of super-level sets as efficient nonconvex approximations of generic semialgebraic sets has been proposed in \cite{DabHen:13}.
In particular, the following optimization problem is considered 
\begin{equation}\label{optp}
\begin{array}{rcll}
v^*_d & := &\displaystyle \min_{p \in \Pd} & \vol\:{\U}(p) \\
&& \mathrm{s.t.} & \K \subseteq {\U}(p).
\end{array}
\end{equation}
The above problem amounts at finding the super-level set that better approximates, in terms of minimum volume, the original set $\K$.
Since $\K$ is compact by assumption,  for problem (\ref{optp})
to have a finite minimum, it was further assumed in  \cite{DabHen:13} that a compact semialgebraic set
$\B$  is given such that $\U(p) \subset \B$ and hence
\[
\U(p)=\{x \in \B : p(x) \geq 1\}.
\]

In this paper, we additionally assume that the set $\B$ is the cartesian product of $n$ one-dimensional sets,
i.e.\ $\B=\B_{1}\times\cdots\times\B_{n}$. For instance, the set $\B$ can be taken as the $n$-dimensional 
 hyper-rectangle 
 \[
\B_{[a,b]}:=\{x \in {\mathbb R}^n : a_i\leq x_i \leq b_i, \: i=1,2,\ldots,n \}.
\]
As noted in \cite[Remark 1]{CePiRe:12}, an outer-bounding box $\B$ of a given semialgebraic set $\K$ can be found
by solving relaxations of the following polynomial optimization problems
 \begin{eqnarray*}
a_{i} = \arg\min_{x\in\mathbb{R}^{n}} x_{i} \text{ s.t. } x \in\K,\quad i = 1,...,n,\\
b_{i} = \arg\max_{x\in\mathbb{R}^{n}} x_{i} \text{ s.t. } x \in\K,\quad i = 1,...,n,
\end{eqnarray*}
which compute the minimum and maximum value of each component of the vector $x$ over the semialgebraic set $\K$.
Note that arbitrarily tight lower and upper bounds can be obtained by means of the 
the techniques discussed e.g.\ in \cite{Lasserre:01,CGTV:03,Parrilo:03}
based on SOS/moment convex relaxations.

The algorithm presented in this paper leverages  some recent results 
presented in \cite{DabHen:13}, which, as a side-result, provide an optimal polynomial approximation of indicator functions.
More precisely, it was shown in that paper that problem  (\ref{optp})
can be approximated by the following convex optimization problem
\begin{equation}\label{l1}
\begin{array}{rcll}
w^*_d & := & \displaystyle \min_{p \in \Pd} &\displaystyle \int_\B p(x)dx \\
&& \mathrm{s.t.} & p \geq 1 \:\:\mathrm{on}\:\: \K\\
&&&p \geq 0 \:\:\mathrm{on}\:\: \B.
\end{array}
\end{equation}
In particular, the following result holds. see \cite[Lemma 1]{DabHen:13}.
\vskip 3mm
\begin{lemma}\label{cvg}
The minimum of problem (\ref{l1}) monotonically
converges from above to the minimum of problem (\ref{optp}), i.e.
$w^*_{d-1} \geq w^*_d \geq v^*_d$ for all $d$, and
$\lim_{d\to\infty} w^*_d = \lim_{d\to\infty} v^*_d$.
\end{lemma}

In particular, the convergence follows from the  fact that the optimal polynomial $p_{d}^{*}$
solution to problem (\ref{l1}) converges in $L^1(\B)$, or equivalently almost uniformly in $\B$,
to the indicator function $\one_{\K}(x)$ as its degree goes to infinity.
This crucial property is the one we will exploit in Section \ref{sec:unifgen} to construct an efficient rejection algorithm
for generating uniform samples in $\K$.

\section{Uniform generation}\label{sec:unifgen}
From the proof of \cite[Lemma 1]{DabHen:13} it can be seen that, for any $d$, the  optimal solution $p_{d}^{*}(x)$ to problem  (\ref{l1})    has the property of being an
\textit{upper approximation} of the indicator function $\one_{\K}(x)$, that is $p_{d}^{*}(x)\ge\one_{\K}(x)$ for all $x\in\B$. Moreover, this approximation becomes tight when $d$ goes to infinity.

Therefore,  this polynomial is a ``dominating density'' of the uniform density $\unif{\K}(x)$ for all $x\in\B$, that is there exists a value $\beta$ such that $\beta p_{d}^{*}(x)\ge \unif{\K}(x)$ for all $x\in\B$.
Hence, the rejection method 
from a dominating density, discussed for instance in \cite[Section 14.3.1]{TeCaDa:13}, can be applied leading to the random sampling procedure described in the following algorithm.

\textsc{Algorithm 1: Uniform Sample Generation in Semialgebraic Set $\K$}

\begin{enumerate}
  \item For a given integer $d>0$, compute the solution of 
\begin{equation}\label{l1p}
\begin{array}{rcll}
p^*_d(x)  & := \displaystyle \arg& \displaystyle\min_{p \in \Pd} &\displaystyle \int_\B p(x)dx \\
&& \mathrm{s.t.} & p \geq 1 \:\:\mathrm{on}\:\: \K\\
&&&p \geq 0 \:\:\mathrm{on}\:\: \B.
\end{array}
\end{equation}
  \item Generate a random sample $\xi^{(i)}$ with density proportional to $p_{d}^{*}(x)$ over $\B$.
  \item If $\xi^{(i)}\not\in\K$ go to step 1.
  \item Generate a sample $u$ uniform on $[0,\,1]$.
  \item If $u\,p_{d}^{*}(\xi^{{i}})\le 1$ return $x^{(i)}=y$, else go to step 1.
\end{enumerate}

\textsc{End of Algorithm 1}

It must be noticed that problem (\ref{l1p}), even though convex and finite-dimensional,
can be very hard to solve. In practice, we solve a tractable LMI problem by strenghtening
the polynomial positivity constraints by polynomial SOS constraints. Convergence
results are however not affected, see \cite{DabHen:13} for more details.

A graphical interpretation of the algorithm  is provided in Figure \ref{fig:oneD}, for the 
case of a simple one-dimensional set 
\[
\K = 
\left\{x\in\mathbb{R}\, : \,
 (x-1)^2-0.5 \ge 0,
x-3\le 0 \right\}.
\]
First, problem (\ref{l1}) is solved (for $d=8$ and $\B=[1.5,\,4]$), yielding the optimal solution
\begin{eqnarray*}
\lefteqn{p_{d}^{*}(x) =
0.069473 x^{8}
      -2.0515 x^{7}+
       23.434 x^{6}
       -139.5 x^{5}+}\\
&&       477.92 x^{4}
      -961.88 x^{3}+
       1090.8 x^{2}
      -606.07 x+
       107.28.
\end{eqnarray*}
As it can be seen, $p_{d}^{*}(x)$ is ``dominating'' the indicator function $\one_{\K}(x)$ for all $x\in\B$.

Then, uniform random samples are drawn in the hypograph of $p_{d}^{*}$. This is done
 by generating uniform samples $\xi^{{i}}$ distributed according to a probability density function (pdf)
proportional to $p_{d}^{*}$ (step 2), 
and then selecting its vertical coordinate uniformly in the interval $[0,\, \xi^{{i}}]$ (step 3).

Finally, if this sample falls below the indicator function $\one_{\K}(x)$ (blue dots) it is accepted, otherwise it is rejected 
(red dots) and the process starts again.

\begin{figure}[ht!]
\centerline{
\includegraphics[width=.6\textwidth]{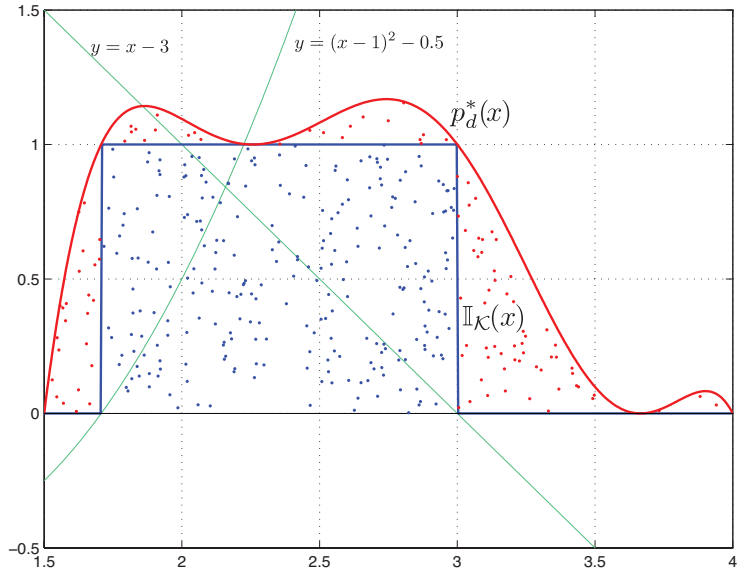}}
\caption{Illustration of the behavior of Algorithm  1  in the one-dimensional case. Blue dots are accepted samples, red dots
are rejected samples. \label{fig:oneD}}
\end{figure}

It is intuitive that this algorithm should outperform classical rejection from the bounding set $\B$, since ``more importance''
is given to the samples inside $\K$, and this importance is weighted by the function $p_{d}^{*}$.

To formally analyze Algorithm 1, we define the
\textit{acceptance rate} (see e.g.\ \cite{Devroye:97}) as one over the  expected number
of samples that have to be drawn from $p_{d}^{*}(x)$ in order
to find one ``good" sample, that is a sample uniformly distributed in $\K$. 
The following results, which is the main theoretical result of the paper, provides the acceptance rate of the proposed algorithm.

\begin{theorem}
Algorithm 1 returns a sample uniformly distributed in $\K$. Moreover, the acceptance rate 
of the algorithm is given by 
\[
\gamma_{d}=\frac{\vol(\K)}{w^*_d},
\]
where $w^*_d$ is the optimal solution of problem (\ref{l1}), that is
\[
w^*_d:=\int_\B p^*_d(x)dx.
\]
\end{theorem}

\begin{proof}
To prove the statement, we first note  that the polynomial $p_{d}^{*}(x)$ defines a density 
\[
f_x(x):= \frac{p_{d}^{*}(x)}{w^*_d}
\]
over $\K$. 
Moreover, by construction, we have $p_{d}^{*}(x)\ge\one_{\K}(x)$, and hence 
\begin{eqnarray}
\frac{1}{w^*_d\vol(\K)} p_{d}^{*}(x) &\ge&\frac{1}{w^*_d\vol(\K)} \one_{\K}(x)\nonumber\\
\frac{1}{\vol(\K)} f_x(x) &\ge& \frac{1}{w^*_d} \unif{\K}(x)\nonumber\\
f_x(x) &\ge& \gamma_{d} \unif{\K}(x) \label{eq:rejrate}.
\end{eqnarray}
Then, it can be immediately seen that Algorithm 1 is a restatement of the classical Von Neumann rejection algorithm,
see e.g.\ \cite[Algorithm 14.2]{TeCaDa:13}, whose acceptance rate is given by the value of $\gamma_{d}$ such that 
(\ref{eq:rejrate}) holds, see for instance \cite{Devroye:86}.
\end{proof}

It follows that the efficiency of the random sample generation increases as $d$ increases, and becomes optimal as $d$ goes to infinity,
as reported in the next corollary.

\begin{corollary}
Let $d$ be the degree of the polynomial approximation of the indicator function of the set $\mathcal{K}$. Then, the acceptance rate tends to one as one increases degree $d$; i.e., 
\[
\lim_{d\to\infty} \gamma_{d} = 1.
\] 
\end{corollary}

Hence, the trade-off is between the complexity of computing a good approximation ($d$ large) on the one hand, and 
having to wait a long time to get a ``good'' sample ($\gamma$ large), on the other hand. Note, however, that the first step can be computed off-line
for a given set $\K$, and then the corresponding polynomial $p_{d}^{*}$ can be used for efficient on-line sample generation.
 
Finally, we highlight that,  in order to apply Algorithm 1 in an efficient way (step 2), a computationally efficient scheme for generating random samples according to a polynomial density is required. This is discussed next.

\section{Generation from a polynomial density} \label{sec:pol_sample}

To generate a random sample according to the multivariate polynomial density $f_x(x)$, one can recur to the conditional density method,
see e.g.\ \cite{Devroye:86}.
This is a recursive method in which the individual entries of the multivariate samples
are generated according to their conditional probability density.
In particular, the  joint pdf of the vector of random variables $x=(x _1\, \cdots\, x _n)$ can be written as
\[
f_x (x_1,\ldots,x_n)= f_{x _1}(x_1)
f_{x _2|x_1}(x_2|x_1)\cdots f_{x _n|x_1\cdots x_{n-1}}
(x_n|x_1\cdots x_{n-1})
\]
where $f_{x _i|x_1,\ldots,x_{i-1}}(x_i|x_1,\ldots,x_{i-1})$ are the conditional
densities. The conditional density is defined as the ratio of marginal densities
\[
f_{x _i|x_1,\ldots, x_{i-1}}(x_i|x_1,\ldots,x_{i-1}) =
\frac{f_{x _1,\ldots,x _{i}}(x_1,\ldots,x_{i})}
{f_{x _1,\ldots,x _{i-1}}(x_1,\ldots,x_{i-1})},
\]
which, in turn, are
given by
\[
f_{x _1,\ldots,x _{i}}(x_1,\ldots,x_{i}) =
\int \cdots \int f_x (x_1,\ldots,x_{n})d  x_{i+1} \cdots
d  x_{n}.
\]

Hence, a random vector $x $ with density $f_x (x)$ can 
be obtained by generating sequentially the $x _i$, $i=1,\ldots,n$,
where $x _i$ is distributed according to the univariate density
$f_{x _i|x_1,\ldots, x_{i-1}}(x_i)$.
The basic idea of this  method is to generate the first random variable
according to $f_{x _1}(x_1)$, then generate the next one conditional on
the first one, and so forth, thus reducing 
an $n$-dimensional generation problem to $n$ one-dimensional problems.
Note that, in the case of polynomial densities, the computation of the marginal densities 
is straightforward, thus making this method particularly appealing.

Moreover, to generate a random sample according to a given univariate polynomial density, the inversion method can be employed, see e.g. \cite[Corollary 14.1]{TeCaDa:13}. This is summarized in the following algorithm 
for the sake of completeness.

\textsc{Algorithm 2: Generation from a univariate polynomial density}
\begin{enumerate}
\item Generate a random variable $w$ uniform on $[p(a_{i}),\,p(b_{i})]$.
\item Compute the unique root $y$ in $[a_{i},b_{i}]$ of the polynomial
\[
y \mapsto \sum_{k=0}^n \frac{a_k}{k+1} y^{k+1}-x.
\]
\item Return $y$.
\end{enumerate}
\textsc{End of Algorithm 2}

This algorithm returns a random variable $x$ distributed
according to the univariate  density proportional to the polynomial $p_x(x)=\sum_{k=0}^n a_k
x^k$ with support $[a_{i},b_{i}]$.

In step 2, the
numerical computation of the root can be performed, up to a given
accuracy, using some standard method such as bisection
or Newton--Raphson. We also remark that more efficient methods
for generating samples from polynomial densities exist,
see for instance the method in~\cite{AhrDie:74}, based on
finite mixtures.

\section{Numerical examples} \label{sec:examples}

\subsection{Approximation of stability region}

As a first illustration of the ideas described in this paper,
we consider the outer approximations by polynomial
super-level sets of the third-degree discrete-time
stability region obtained by solving the convex
optimization problem in Step 1 of Algorithm 1.
Tightness of these approximations is crucial for
a good performance of the sample generation method.

A third-degree monic polynomial $z \in {\mathbb C} \mapsto x_1+x_2z+x_3z^2+z^3$
with real coefficients is discrete-time stable (i.e. all its roots lie in the interior of the unit disk
of the complex plane) if and only if its coefficient vector $x=(x_1,x_2,x_3)$ belongs
to the interior of the basic semialgebraic set
\[
\begin{array}{rcl}
\K = \{x \in {\mathbb R}^3 & : & g_1(x)=3-3x_1-x_2+x_3 \geq 0, \\
&& g_2(x) = 1-x_1+x_2-x_3 \geq 0,\\
&& g_3(x) = 1-x_2-x_1^2+x_1x_3 \geq 0\}
\end{array}
\]
which is nonconvex, see e.g. \cite[Example 4.3]{HeLa:12}.
Set $\K$ in included in the bounding box $\B = [-1,3]\times[-3,3]\times[-1,1]$.
In Figure \ref{fig:schur3outer4} we represent the super-level set $\U(p_4)=\{x : p_4(x)\geq 1\}$
of the degree 4 polynomial solving optimization problem (\ref{l1p}) with
the polynomial positivity constraints replaced with polynomial SOS constraints.
The set $\U(p_4)$ is a guaranteed outer approximation of $\K$.
In Figure \ref{fig:schur3outer8} we represent the much tighter degree 8
outer approximation $\U(p_8) \supset K$.

 \begin{minipage}{\linewidth}
 \centering
 \begin{minipage}{0.45\linewidth}
 \begin{figure}[H]
\includegraphics[width=\textwidth]{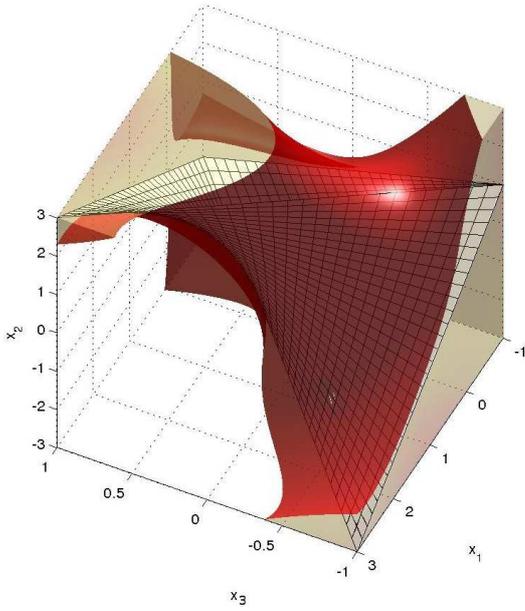}
\caption{Degree 4 outer polynomial approximation (boundary in red, interior in light red)
of the third degree discrete-time stability region (inner volume in white).\label{fig:schur3outer4}}
\end{figure}
\end{minipage}
\hspace{0.05\linewidth}
\begin{minipage}{0.45\linewidth}
\begin{figure}[H]
\includegraphics[width=\textwidth]{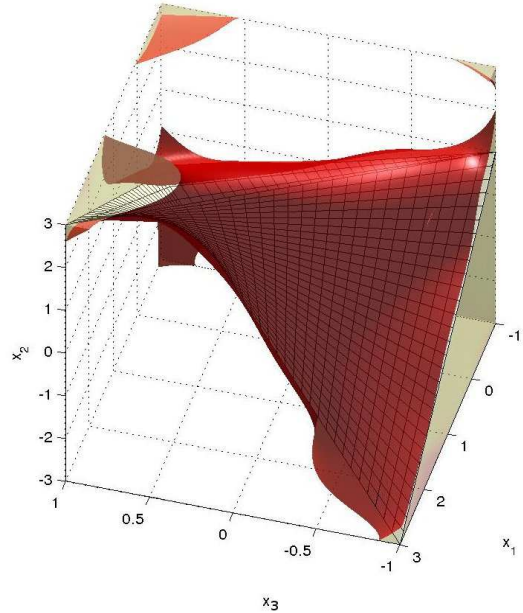}
\caption{Degree 8 outer polynomial approximation  (boundary in red, interior in light red)
of the third degree discrete-time stability region (inner volume in white).\label{fig:schur3outer8}}
\end{figure}
\end{minipage}
\end{minipage}


\subsection{Approximation of stabilizability region}

As another control-oriented illustration of the polynomial super-level
set approximation used by our sampling Algorithms 1 and 2, consider
\cite[Example 4.4]{HeLa:12} which is a degree 4
discrete-time polynomial $z \in {\mathbb C} \mapsto 
x_2+2x_1z-(2x_1+x_2)z^3+z^4$
to be stabilized with two real control parameters $x_1, x_2$.
In other words, we are interested in sampling uniformly in the set $\K$
of values of $(x_1,x_2)$ such that this polynomial has its roots with
modulus less than one. An explicit basic semialgebraic description
of the sampling set is
\[\small
\begin{array}{l}
\K = \{x \in {\mathbb R}^3 \: :\: g_1(x)=1+2x_2\geq 0, \\
\quad g_2(x) =2-4x_1-3x_2 \geq 0,\\
\quad g_3(x) =10-28x_1-5x_2-24x_1x_2-18x^2_2 \geq 0,\\
\quad g_4(x) = 1-x_2-8x_1^2-2x_1x_2-x_2^2-8x_1^2x_2-6x_1x_2^2 \geq 0\}.
\end{array}
\]
This set is nonconvex and it is included in the box
$\B=[-1,1]^2$. In Figure \ref{fig:schur4} we represent the graph of the
degree 10 polynomial $p_{10}(x)$ constructed by solving optimization problem (\ref{l1p}) with
the polynomial positivity constraints replaced with polynomial SOS constraints.
From this we get the outer approximation $\K \subset \U(p_{10})=\{x \in {\mathbb R}^2 \: : \: p_{10}(x) \geq 1\} $
used in our sampling algorithm.
\begin{figure}[!ht]
\centerline{
\includegraphics[width=.6\textwidth]{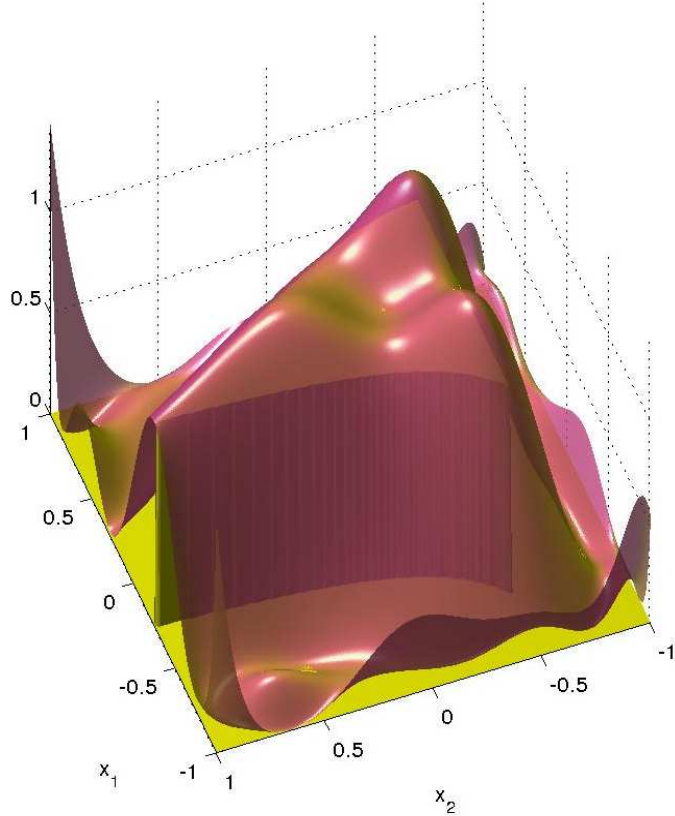}
}
\caption{Degree 10 polynomial approximation (surface in pink)
of the indicator function of the nonconvex planar stabilizability
region.\label{fig:schur4}}
\end{figure}

\subsection{Sampling in a nonconvex semialgebraic set}

To demonstrate the behavior of Algorithms 1 and 2,  we revisit a numerical example originally introduced in \cite{CePiRe:12}.
The considered semialgebraic set $\K$ is the two-dimensional nonconvex region described as:
\begin{eqnarray}
\label{K-es}
\lefteqn{\K:= 
\left\{(x_{1},x_{2})\in\mathbb{R}^{2}\, :\quad\right.}\\
&&\left.
(x_{1} -1)^{2} +(x_{2} -1)^{2} \le 1,\:\: x_{2}  \le 0.5 x^2_{1} \right\}.
\nonumber
\end{eqnarray}
In \cite{CePiRe:12}, the outer-bounding box 
\[
\B = \left\{(x_{1} , x_{2} )\in\mathbb{R}^{2}\, :\, 0.46 \le
x_{1} \le  2.02,\:\: 0.03 \le x_{2} \le 1.64\right\}
\]
was considered. 
Figure \ref{fig:cerone8} shows a two-dimensional plot of the indicator function $\one_\K(x)$, and the corresponding optimal solution $p_{d}^{*}(x)$ for $d=8$. The results of Algorithm 1 are reported in Figure \ref{fig:cerone8samples}. The red points represent the points which have been discarded. To this regard, it is important to notice that also some point falling inside $\K$ has been rejected. 
This is fundamental to guarantee uniformity of the discarded points.
\begin{figure}[!ht]
\centerline{
\includegraphics[width=.6\textwidth]{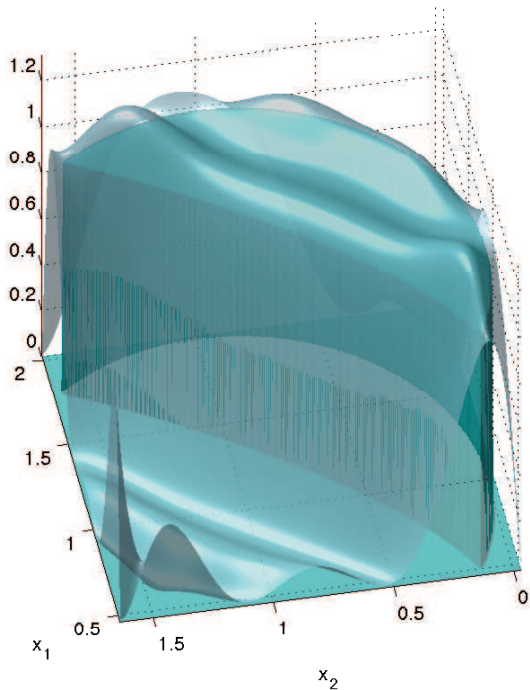}
}
\caption{Optimal polynomial approximation of degree 8 of the indicator function.\label{fig:cerone8}}
\end{figure}

\begin{figure}[!ht]
\centerline{
\includegraphics[width=.6\textwidth]{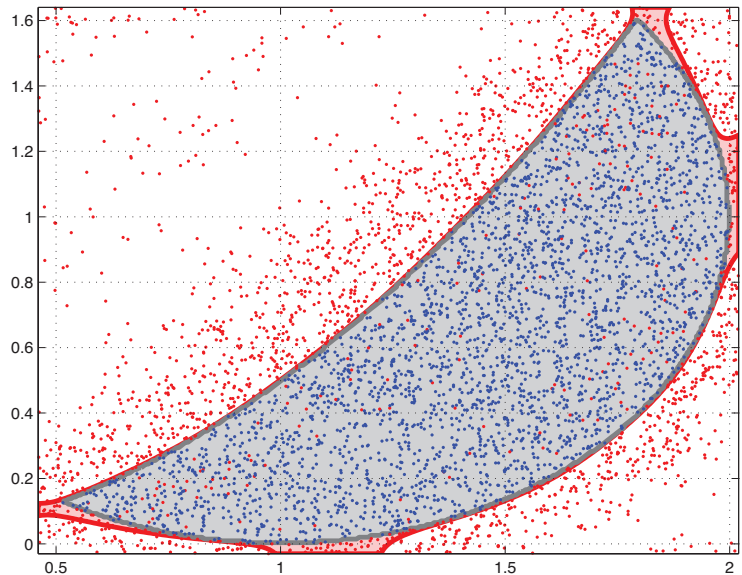}
}
\caption{Uniform random samples generated according to Algorithms 1 and 2.
The grey area is the set $\K$ defined in (\ref{K-es}), the pink area is the superlevel set 
$\U(p_d^*)$. The red dots are the discarded samples. The remaining samples (blue) are uniformly distributed inside $\K$.  \label{fig:cerone8samples}}
\end{figure}

\section{Concluding remarks} \label{sec:conclusion}
In this paper, a numerically efficient procedure for generating ``truly'' random samples inside a given
(possibly nonconvex) semialgebraic set $\K$ has been proposed. The algorithm is based on an acceptance/rejection
scheme constructed upon an optimal polynomial superlevel set guaranteed to contain $\K$. 
A key feature of the method is that, for a given set $\K$ and polynomial degree $d$, 
this approximation, which is undoubtedly the most time-consuming step of the sampling scheme,  
can be constructed \textit{a priori} and once for all, and then be used for online generation.
The rejection rate is shown to become asymptotically optimal when the degree of the polynomial approximation increases. 
Future work will concentrate on the application of this  to specific fixed-order controller design problems.

\section*{Acknowledgement}

This work was partially supported by a bilateral project CNR-CNRS.

\end{document}